\documentclass[twocolumn]{autart}     

\usepackage{graphics}
\usepackage{array, amssymb, amscd}
\usepackage[cmex10]{amsmath}
\usepackage{enumerate}
\usepackage{savesym}
\savesymbol{AND}

\usepackage{stfloats}
\usepackage{url}
\usepackage{cite} 
\usepackage[geometry]{ifsym}
\usepackage{bbm}
\usepackage{algorithm}
\usepackage{algorithmic}
\usepackage{subfig}
\usepackage{graphicx}
\usepackage{color}
\usepackage[dvipsnames]{xcolor}
\usepackage{verbatim}
\usepackage{dsfont}
\usepackage{enumitem}

\let\oldFootnote\footnote
\newcommand\nextToken\relax

\renewcommand\footnote[1]{%
    \oldFootnote{#1}\futurelet\nextToken\isFootnote}

\newcommand\isFootnote{%
    \ifx\footnote\nextToken\textsuperscript{,}\fi}

\newtheorem{definition}{Definition}
\newtheorem{assumption}{Assumption}
\newtheorem{proposition}{Proposition}

\newtheorem{theorem}{Theorem}
\newtheorem{proof}{Proof}

\newtheorem{remark}{Remark}

\usepackage[colorlinks=true,
        raiselinks=true,
        linkcolor=MidnightBlue,
        citecolor=Mahogany,
        urlcolor=ForestGreen,
        pdfauthor=,
        pdftitle={Author Reply},
        pdfsubject={Author Reply},
        plainpages=false]{hyperref}



\begin{document}
\begin{frontmatter}

\title{Price of anarchy in electric vehicle charging control games: When Nash equilibria achieve social welfare\thanksref{footnoteinfo}}

\thanks[footnoteinfo]{Research was supported by the European Commission, H2020, under the project UnCoVerCPS, grant number 643921, by EPSRC UK under the
grant EP/P03277X/1, and by a MathWorks professorship support. Preliminary results related to Sections 3.2 and 3.3 of the current manuscript can be found in \cite{Deori_ifac}. Corresponding author: Luca Deori.}

\author[DEIB]{Luca Deori}\ead{luca.deori@polimi.it},
\author[OX]{Kostas Margellos}\ead{kostas.margellos@eng.ox.ac.uk},
\author[DEIB]{Maria Prandini}\ead{maria.prandini@polimi.it}

\thanks[DEIB]{Dipartimento di Elettronica, Informazione e Bioingegneria, Politecnico di Milano, Piazza Leonardo da Vinci 32, 20113 Milano, Italy.}

\thanks[OX]{Department of Engineering Science, University of Oxford,
Parks Road, OX1 3PJ, Oxford, UK}


\begin{abstract}
We consider the problem of optimal charging of plug-in electric vehicles (PEVs). We treat this problem as a multi-agent game, where vehicles/agents are heterogeneous since they are subject to possibly different constraints. Under the assumption that electricity price is affine in total demand, we show that, for any finite number of heterogeneous agents, the PEV charging control game admits a unique Nash equilibrium, which is the optimizer of an auxiliary minimization program.
We are also able to quantify the asymptotic behaviour of the \emph{price of anarchy} for this class of games. More precisely, we prove that if the parameters defining the constraints of each vehicle are drawn randomly from a given distribution, then,  the value of the game converges almost surely to the optimum of the cooperative problem counterpart as the number of agents tends to infinity.
In the case of  a discrete probability distribution, we provide a systematic way to abstract agents in homogeneous groups and show that, as the number of agents tends to infinity, the value of the game tends to a deterministic quantity.
\end{abstract}

\end{frontmatter}

\section{Introduction} \label{sec:secI}
Electric vehicles obtain some or all of their energy from the electricity grid, and are typically referred to as plug-in electric vehicles (PEVs). Their penetration is expected to increase significantly, since, not only they contribute to pollution reduction, but, by charging over low electricity price periods, they also serve as virtual dynamic storage, contributing to the stability of the electric grid (see \cite{Rahman_Shrestha_1993,Denholm_Short_2006,Callaway_Hiskens_2011,Zhang_2014}).
In an electric vehicle charging control context two cases can be distinguished. The first case refers to a set-up where vehicles are social welfare maximizing entities and cooperate in view of minimizing the overall population cost. Under this setting, \cite{Gan_etal_2013,Deori_etal_2016a,Deori_etal_2016b} propose iterative schemes that involve every vehicle solving a local minimization program, and show convergence to the social welfare optimum.
In the second case vehicles act as selfish agents that seek to minimize their local cost, without being concerned with social welfare paradigms. This gives rise to multi-agent non-cooperative games, and the main concern is the computation of Nash equilibrium strategies.
A complete theoretical analysis is provided in \cite{Caines_etal_2007,Lions_etal_2007} for stochastic continuous-time problems, but in the absense of constraints. The deterministic, discrete-time problem variant, was investigated in
\cite{Hiskens_2013}, and was further extended in \cite{Franc_2014,Grammatico_etal_2015} to account for the presence of constraints.
However, for any finite number of agents, an approximate Nash equilibrium is computed, while the exact Nash one is reached only in the limiting case where the number of agents tends to infinity. The recent work of \cite{Paccagnan_etal_2016} overcomes this issue under the assumption that vehicles are aware of the way the total population consumption affects the price that drives their behaviour.

One challenge associated with the aforementioned stream of literature is that there is no common awareness on how the resulting Nash equilibrium solution is related to the associated social welfare optimum.
In this paper we follow a pricing set-up similar to the seminal paper by \cite{Arrow_Debreu_1954}, and account for constraint heterogeneity by assuming that the parameters defining the constraints of each vehicle are drawn randomly from a given distribution. We consider a multi-stage variant of the problem, however,  we assume the price is an affine function of the total consumption.
Under this set-up, our paper provides the following contributions: 

(1) We quantify, to the best of our knowledge for the first time, the limiting value of the \emph{price of anarchy} \cite{Koutsoupias_Papadimitriou_1999} for this class of games. The price of anarchy provides the means to quantify the efficiency of Nash equilibria, and is defined as the ratio between the worst-case value of the game achieved by a Nash equilibrium (in our setting there is a unique one) and the social optimum. We prove that as the number of agents tends to infinity this ratio tends to one for almost any choice of the random heterogeneity parameters (Theorem \ref{thm:limit_Nash_opt}). This result extends \cite{Hiskens_2013} to the case of heterogeneous agents that are subject to constraints, without resorting to approximate Nash equilibria and primal-dual algorithms as in \cite{Li_Zhang_2016}.
As a byproduct we show that, for any finite number of possibly heterogeneous agents, the PEV charging control game admits a unique Nash equilibrium, which is the minimizer of an auxiliary minimization program (Proposition \ref{prop:Nash_opt}). This is due to the fact that the underlying game is potential \cite{facchinei2011PG}, however, our proof line is different and is based on fixed-point theoretic results. This result opens the road for the use of iterative algorithms for decentralized computation of Nash equilibria \cite{Gan_etal_2013,Paccagnan_etal_2016,Deori_etal_2016b}.

(2) We provide the discrete time counterpart of the mean-field game theoretic approach in \cite{Caines_etal_2007}, treating heterogeneity in a probabilistic manner, thus complementing the deterministic approaches of \cite{Grammatico_etal_2015,Paccagnan_etal_2016,Li_Zhang_2016}.
In particular, we show that if the distribution of the random parameters that render agents' constraints heterogeneous is discrete, agents can be abstracted in homogeneous groups and, for almost any realization of the random heterogeneity parameters, as the number of agents tends to infinity, the value of the game tends to a deterministic quantity (Theorem \ref{thm:heter}).

It should be noted that our set-up exhibits similarities with multi-participant market investigations in \cite{Caramanis_2011,Caramanis_2012,Dahleh_2015,Caramanis_IEEE}. In particular, it is shown in \cite{Caramanis_2012} that under current day-ahead operations participants have the incentive to self-dispatch, and the resulting social welfare market clearing prices are not practically viable. This is not in contrast with our results, since we show that Nash equilibria and social optima tend to coincide only in the limiting case of an infinite number of agents, and may differ for finite populations. Moreover, we consider a stylized architecture without including a distribution network model.

Section \ref{sec:secII} introduces the non-cooperative PEV charging control game and its social welfare counterpart. Section \ref{sec:secIII} quantifies the price of anarchy for the limiting case of an infinite number of agents. In Section \ref{sec:secIV}, we investigate the effect heterogeneity has in the value of the game, while Section \ref{sec:secVI} provides some directions for future work.

\section{Electric vehicle charging control problem} \label{sec:secII}
\subsection{Cooperative set-up} \label{sec:secIIA}
{We first consider the case of $m$ PEVs that seek to determine their charging profile along some discrete time horizon $[0,h-1]$ of arbitrary  length $h \in \mathbb{N}$ so as to minimize the total charging cost for the entire fleet. This corresponds to a cooperative set-up that is likely to occur when vehicles belong to the same managing entity.
To this end, let $H = \{0,1,\ldots,h-1\}$ and $I = \{1,\ldots,m\}$}. 
Consider the following optimization program:
\begin{align}
\min_{\{ x^{it} \in \mathbb{R} \}_{\substack{t \in H\\i \in I }}}
&\sum_{t \in H} p^t \Big ( \sum_{i \in I} x^{it} + x^{0t} \Big )^2  \label{eq:prob_SO_obj}\\
\text{subject to:~ } &\sum_{t \in H} x^{it} = \gamma^i, \text{ for all } i \in I, \label{eq:prob_SO_gamma} \\
& x^{it} \in [\underline{x}^{it}, \overline{x}^{it}], \text{ for all } t \in H,~ i \in I, \label{eq:prob_SO_limits}
\end{align}
where $x^{it} \in \mathbb{R}$ is the charging rate of vehicle $i$, $i \in I$, at time $t$, $t \in H$, and $p^t \geq 0$ is an electricity price coefficient at time $t$.
For each $t \in H$, we denote by $x^{0t} \geq 0$ the non-PEV demand which, for a fixed number of PEVs $m$, is treated as constant and not as an optimization variable in the optimization programs below. Similarly to \cite{Hiskens_2013,Franc_2014},  for all $t \in H$, we assume that $\lim_{m \to \infty} x^{0t}/m = \hat{x}^{0t}$ is constant, allowing the non-PEV demand to grow linearly in the number of agents $m$ if $\hat{x}^{0t} \neq 0$.

The price of electricity is given by $p^t (\sum_{i \in I} x^{it} + x^{0t})$, and is assumed to depend linearly on the total PEV and non-PEV demand through $p^t$. Dependency of price on the PEV demand is affine due the presence of $x^{0t}$.
Our choice for an affine price function is a simplification over
\cite{Arrow_Debreu_1954,Gan_etal_2013,Hiskens_2013} where convex monotone increasing functions are allowed, and is motivated by \cite{Grammatico_etal_2015}, where an affine function is also employed, as well as by the numerical investigations of  \cite{Gharesifard_etal_2016} (in the corresponding theoretical analysis more general functions are allowed). The slope of this function encodes the inverse of the price elasticity of demand, and is motivated by the fact that marginal prices in lossless unconstrained energy systems are affine functions of the total production/demand \cite{Caramanis_IEEE}.
The objective function in \eqref{eq:prob_SO_obj} encodes the total electricity cost over $[0,h-1]$.   Constraint \eqref{eq:prob_SO_gamma} represents a prescribed charging level $\gamma^i \in \mathbb{R}$, $\gamma^i > 0$, to be reached by each vehicle $i$ at the end of the considered time horizon $H$, whereas \eqref{eq:prob_SO_limits} imposes minimum ($\underline{x}^{it} \in \mathbb{R}$, $\underline{x}^{it} \geq 0$) and maximum ($\overline{x}^{it} \in \mathbb{R}$, $\overline{x}^{it} < \infty$) limits, respectively, on $x^{it}$.

For all $i \in I$, let $x^i = [x^{i0}, \ldots, x^{i(h-1)}]^\top \in \mathbb{R}^{|H|}$, where $|\cdot|$ denotes the cardinality of its argument. Let also
$f: \mathbb{R}^{|H|} \times \mathbb{R}^{m|H|} \to \mathbb{R}$ be such that, for all $i \in I$, for any $(x^i,x^{-i}) \in \mathbb{R}^{m|H|}$,
\begin{align}
f(x^i,x^{-i}) = \sum_{t \in H} x^{it} p^t \Big ( \sum_{\substack{j \in I\\ j \neq i}} x^{jt} +x^{it} +  x^{0t} \Big ), \label{eq:agent_payoff}
\end{align}
where by $x^{-i} \in \mathbb{R}^{(m-1)|H|}$ we imply a vector including the decision variables of all vehicles except vehicle $i$ (recall that $x^{0t}$ is constant for any fixed $m$ and hence not included in these vectors).
Moreover, for all $i \in I$, let
\begin{align}
X^i = \big \{ x^i \in \mathbb{R}^{|H|}:~ & \sum_{t \in H} x^{it} = \gamma^i \text{ and } \nonumber \\
&x^{it} \in [\underline{x}^{it}, \overline{x}^{it}], \text{ for all } t \in H \big \}, \label{eq:agent_con}
\end{align}
denote the constraint set corresponding to vehicle $i$.
Let $x = (x^1,\ldots,x^m)$ and $X = X^1 \times \ldots \times X^m$, and consider $f_0: \mathbb{R}^{m|H|} \to \mathbb{R}$ such that
$f_0(x) = \sum_{t \in H} x^{0t} p^t ( \sum_{j \in I} x^{jt}  + x^{0t} )$,
which represents the cost of non-PEV demand.
{We can then rewrite \eqref{eq:prob_SO_obj}-\eqref{eq:prob_SO_limits} as }
\begin{align}
\mathcal{P}:~ \min_{\{x^i \in X^i\}_{i \in I}} f_0 (x) + \sum_{i \in I} f(x^i,x^{-i}). \label{eq:opt_P}
\end{align}
and refer to its optimal solution as social optimum.
Note that local utility functions that depend only on the decision vector $x^i$ of each vehicle $i$, $i \in I$, and are possibly different per vehicle, can be incorporated in $\mathcal{P}$ by means of an epigraphic reformulation (see \cite{Deori_etal_2016a}).

\begin{assumption} \label{ass:feas}
{Fix any $m \geq 1$ and let $\gamma^i > 0$, $i \in I$. \\
a) The sets $X^i$, $i \in I$, are nonempty and compact. \\
b) The price coefficient satisfies $p^t > 0$, for all $t \in H$.}
\end{assumption}
The second part of Assumption \ref{ass:feas} is only needed for the proof of Theorem \ref{thm:limit_Nash_opt}, but is naturally satisfied in situations of practical relevance.

Denote the set of social optima $M$ of $\mathcal{P}$ by
\begin{align}
M = \arg \min_{ \{x^i \in X^i\}_{i \in I}} f_0(x) + \sum_{i \in I} f(x^{i},x^{-i}). \label{eq:def_min}
\end{align}
Note that \eqref{eq:def_min} involves minimizing a continuous function (as an effect of being convex), over a compact set (which is convex) due to Assumption \ref{ass:feas}. As such, the minimum is achieved due to Weierstrass' theorem in \cite[Proposition A.8, p. 625]{Bertsekas_Tsitsiklis_1997}. Under a similar reasoning all subsequent minimization problems are well defined.  It should be emphasized that $f_0$ is introduced to facilitate the compact representation of \eqref{eq:prob_SO_obj} in \eqref{eq:opt_P} and captures the cost of non-PEV demand, which does not appear in the gaming formulation of the next subsection where, similarly to \cite{Hiskens_2013,Franc_2014}, agents' pay-off functions are given by \eqref{eq:agent_payoff}.

\subsection{Non-cooperative set-up} \label{sec:secIIB}
We now consider the case where the $m$ vehicles act in a non-cooperative manner. In particular, each vehicle/agent $i$, $i \in I$, aims at determining a charging profile $x^i$ that minimizes its pay-off function $f(x^i,x^{-i})$, as this is given by \eqref{eq:agent_payoff}, which depends on its own decision vector $x^i$ and on the other agents decision vector $x^{-i}$, subject to a local constraint $x^i \in X^i$. We say that for all $i$, $i \in I$, the tuple $(x^i,x^{-i})$ is a Nash equilibrium of the game, if each agent $i$, given the strategies $x^{-i}$ of the other agents, has no interest in changing its own strategy $x^i$.

\begin{definition}
For all $i \in I$, each agent $i$ has a pay-off function $f(\cdot,x^{-i})$
and a constraint set $X^i$. The set of Nash equilibria $N$ of the non-cooperative game is given by
\begin{align}
N = \big\{x\in X:\ f(x^{i},x^{-i}) &\le f(\zeta^i,x^{-i}) \nonumber \\ &\text{for all } \zeta^i \in X^i,~ i \in I \big\}, \label{eq:def_Nash}
\end{align}
{where $x = (x^1,\ldots,x^m)$ and $X = X^1 \times \ldots \times X^m$.}
\end{definition}
Since each agent has a pay-off function of the same structure, the resulting game is a potential game \cite{facchinei2011PG,Voorneveld2000PG}.

\section{Nash equilibria versus social optima} \label{sec:secIII}
\subsection{Nash equilibria as fixed-points} \label{sec:secIIIA}
The results of this subsection {do not require the pay-off function to exhibit the form of \eqref{eq:agent_payoff} and are more general;} in fact each agent could have a different pay-off function, convex with respect to the decision vector of the particular agent, but possibly non-differentiable.

For each $i$, $i \in I$, consider the mappings $T^i:X \to X^i$ and $\widetilde{T}^i: X \to X^i$, defined such that, for any $x \in X$,
\begin{align}
T^{i}(x)= &\arg\min_{z^i\in X^i}  \|z^i -x^i\|^2 \label{eq:T_i} \\
&\text{subject to } \nonumber \\
&f(z^i,x^{-i}) \le \min_{\zeta^i \in X^i} f(\zeta^i,x^{-i}),  \nonumber \\
\widetilde{T}^i(x) = &\arg\min_{z^i \in X^i} f(z^i,x^{-i}) + c \|z^i - x^i\|^2, \label{eq:Ttilde_i}
\end{align}
for any $c>0$.
Note that both mappings are well defined since both the minimizers of \eqref{eq:T_i} and \eqref{eq:Ttilde_i} are unique. As for the mapping in \eqref{eq:T_i}, a tie-break rule is implemented to select, in case $f(\cdot,x^{-i})$ admits multiple minimizers over $X^i$, the one closer to $x^i$ with respect to the Euclidean norm. In contrast, the mapping $\widetilde{T}^i$ in \eqref{eq:Ttilde_i} includes in the objective function an additional term weighted by $c > 0$, which penalizes the deviations from the current decision vector $x^i$ and makes it strictly convex.
Notice that, with a slight abuse of notation, by $T^{i}(x)$ and $\widetilde{T}^i(x)$, we imply the minimizers of \eqref{eq:T_i} and \eqref{eq:Ttilde_i}, respectively, and not the corresponding (singleton due to uniqueness) sets.

Define also the mappings $T: X \to X$ and $\widetilde{T}: X \to X$, such that their components are given by $T^i$ and $\widetilde{T}^i$, respectively, for $i \in I$, i.e., $T = ( T^1,\ldots,T^m )$ and $\widetilde{T} = ( \widetilde{T}^1,\ldots,\widetilde{T}^m )$. They can be equivalently written as
\begin{align}
T(x) = &\arg \min_{z \in X} \sum_{i \in I} \|z^i - x^i\|^2\ \label{eq:T} \\
&\text{subject to } \nonumber \\
&f(z^i,x^{-i}) \le \min_{\zeta^i \in X^i} f(\zeta^i,x^{-i}), ~\forall i \in I, \nonumber\\
\widetilde{T}(x) = &\arg\min_{z \in X} \sum_{i \in I} \big [ f(z^i,x^{- i}) + c\|z^i-x^i\|^2 \big ].\label{eq:Ttilde}
\end{align}
The set of fixed points for $T$ and $\widetilde{T}$ is given by
\begin{align}
F_T &= \big \{ x \in X:~ x=T(x)\big \}\label{eq:FP_T}, \\
F_{\widetilde{T}} &= \big \{ x \in X:~ x=\widetilde{T}(x)\big \}.\label{eq:FP_Ttilde}
\end{align}

We first show that the set of Nash equilibria $N$ and the set of fixed-points $F_T$ of the mapping $T$ in \eqref{eq:T} coincide.
\begin{proposition}\label{prop:N_FT}
Under Assumption \ref{ass:feas}.a), $N = F_T$.
\end{proposition}
\begin{proof}
1) $N \subseteq F_T$: Fix any $x \in N$. For each $i \in I$, denote $x$ by $(x^i,x^{-i})$. The fact that $x \in N$ implies that $x^i$ is a minimizer of $f(\cdot,x^{-i})$, for all $i \in I$, indeed according to \eqref{eq:def_Nash}, $f(x^i,x^{-i})$ will be no greater than the values that $f$ may take if evaluated at $(\zeta^i,x^{-i})$, for any $\zeta^i \in X^i$, i.e., $f(x^i,x^{-i}) \leq f(\zeta^i,x^{-i})$, for all $\zeta^i \in X^i$.
The last statement can be equivalently written as $f(x^i,x^{-i}) \leq \min_{\zeta^i \in X^i} f(\zeta^i,x^{-i})$ which means that $x$ satisfies the inequality in \eqref{eq:T}. Moreover, $x$ is also optimal for the objective function in \eqref{eq:T}, since it results in zero cost. Hence, $x=T(x)$, which by \eqref{eq:FP_T} implies that $x \in F_T$.\\
2) $F_T \subseteq N$: Fix any $x \in F_T$. By the definition of $F_T$, and due to the inequality in \eqref{eq:T} that is embedded in the definition of $T$, we have that for all $i \in I$, $f(x^i,x^{-i}) \leq \min_{\zeta^i \in X^i} f(\zeta^i,x^{-i})$. The last statement implies that $x^i$ is a minimizer of $f(\cdot,x^{-i})$ over $X^i$, and hence $f(x^{i},x^{-i}) \le f(\zeta^i,x^{-i})$, $\forall \zeta^i \in X^i$, $\forall i \in I$, which due to \eqref{eq:def_Nash} implies that $x\in N$.
\hfill $\Box$
\end{proof}

We next show that the set of fixed-points $F_T$ of $T$ in \eqref{eq:FP_T} and the set of fixed-points $F_{\widetilde{T}}$ of $\widetilde{T}$ in \eqref{eq:FP_Ttilde} coincide.

\begin{proposition}\label{prop:FT_FTilde}
Under Assumption \ref{ass:feas}.a), $F_T = F_{\widetilde{T}}$.
\end{proposition}
\begin{proof}
1) $F_T \subseteq F_{\widetilde{T}}$: Fix any $x \in F_T$. By the definition of $F_T$, and due to the inequality in \eqref{eq:T} that is embedded in the definition of $T$, we have that for all $i \in I$, $f(x^i,x^{-i}) \leq \min_{\zeta^i \in X^i} f(\zeta^i,x^{-i})$. The last statement implies that $x^i$ is a minimizer of $f(\cdot,x^{-i})$ over $X^i$, and hence $f(x^{i},x^{-i}) \le f(\zeta^i,x^{-i})$, $\forall \zeta^i \in X^i$, $\forall i \in I$. Therefore, we would also have that $f(x^{i},x^{-i}) \le f(\zeta^i,x^{-i}) + c \| \zeta_i - x_i \|^2$, $\forall \zeta^i \in X^i$, $\forall i \in I$. The latter, due to \eqref{eq:Ttilde_i} implies that $x^i = T^i(x)$, for all $i \in I$, and hence $x \in F_{\widetilde{T}}$.

2) $F_{\widetilde{T}} \subseteq F_T$: Fix any $x \in F_{\widetilde{T}}$. By the definition of $F_{\widetilde{T}}$, and due to \eqref{eq:Ttilde_i}, the latter implies that $x^i = \widetilde{T}^i(x)$ for all $i \in I$. We thus have that, for all $i \in I$,
\begin{align}
f(x^{i},x^{-i}) \leq f(\zeta^i,x^{-i}) +c\|\zeta^i&-x^i\|^2, ~\forall \zeta^i \in X^i. \label{eq:proof_FT1}
\end{align}
If in addition $x^i$ minimizes $f(\cdot,x^{-i})$ over $X^i$, for all $i \in I$, then $x^i$ would satisfy the inequality in \eqref{eq:T_i}, while resulting in zero cost. We would thus have that $x^i = T^i(x)$, for all $i \in I$, and hence $x \in F_T$.

To show that, for all $i \in I$, $x^i$ minimizes $f(\cdot,x^{-i})$ over $X^i$, assume for the sake of contradiction that this is not the case and there exists
$z^i \in X^i$, $z^i \neq x^i$, such that $f(z^i,x^{-i}) < f(x^i,x^{-i})$.
For any $\alpha \in (0,1)$, let $\zeta^i=\alpha z^i +(1-\alpha)x^i$. Note that by convexity of $X^{i}$, $\zeta^i \in X^i$, whereas by convexity of $f(\cdot,x^{-i})$ with respect to its first argument we have that
\begin{align}
f(\zeta^i,x^{-i}) \le \alpha f(z^i,x^{-i}) + (1-\alpha)f(x^{i},x^{-i}), \label{eq:proof_FT2}
\end{align}
which, by rearranging some terms, can be rewritten as
\begin{align}
f(\zeta^i,x^{-i}) +\alpha \big ( f(x^{i},x^{-i})-f(&z^i,x^{-i}) \big ) \nonumber \\
&\le f(x^{i},x^{-i}). \label{eq:proof_FT3}
\end{align}

Note that, since $f(x^{i},x^{-i})-f(z^i,x^{-i}) >0$, $\|z^i-x^i\| >0$ and $c>0$, there exists $\alpha \in (0,1)$ such that
\begin{align}
\alpha (f(x^{i},x^{-i})-f(z^i,x^{-i})) &> c\alpha^2\|z^i-x^i\|^2 \nonumber \\&=c\|\zeta^i-x^i\|^2, \label{eq:proof_FT4}
\end{align}
where the equality follows from the definition of $\zeta^i$ (note that $\zeta^i$ depends on the choice of $\alpha$). By \eqref{eq:proof_FT3} and \eqref{eq:proof_FT4} we have that there exists $\alpha$ such that
\begin{align}
f(\zeta^i,x^{-i}) + c\|\zeta^i-x^i\|^2 < f(x^{i},x^{-i}). \label{eq:proof_FT5}
\end{align}
The last statement, together with \eqref{eq:proof_FT1}, leads to a contradiction, showing that $x^i$ minimizes $f(\cdot,x^{-i})$ over $X^i$. \hfill $\Box$
\end{proof}

An alternative proof for a result similar to Proposition \ref{prop:FT_FTilde} was provided in \cite[Proposition 3]{Deori_etal_2016b}, relying, however, on the additional assumption that the objective functions involved are differentiable.
The following corollary is a direct consequence of Propositions \ref{prop:N_FT} and \ref{prop:FT_FTilde}.
\begin{cor}\label{cor:N_FTilde}
Under Assumption \ref{ass:feas}.a), $N=  F_{\widetilde{T}}$.
\end{cor}

\subsection{Nash equilibria as social optima of an auxiliary problem} \label{sec:secIIIB}
We show that the set of Nash equilibria $N$ defined in \eqref{eq:def_Nash} coincides with the set of optimizers of an auxiliary minimization program. To this end, for all $i \in I$, let
\begin{align}
\mathcal{P}_a:~ \min_{\{x^i \in X^i\}_{i \in I}}  f_0(x) + \sum_{i \in I} \big [ f(x^i,x^{-i}) + f_a(x^i) \big ], \label{eq:opt_Pa}
\end{align}
where $f_a(x^i) = \sum_{t \in H} p^t (x^{it})^2$. Problem $\mathcal{P}_a$ is a centralized convex optimization program. Let $\widetilde{T}_a = \big ( \widetilde{T}^1_a,\ldots,\widetilde{T}^m_a \big)$ (see also equation (5) in \cite{Deori_etal_2016b}), where, for all $i \in I$, for any $c>0$,
\begin{align}
\widetilde{T}^i_a(x) = \arg &\min_{z^i \in X^i} f_0(z^i,x^{-i}) + f(z^i,x^{-i}) + f_a(z^i) \nonumber \\
&+ \sum_{\substack{k \in I\\ k \neq i}} \big [ f(x^k,(z^i,x^{-\{k,i\}})) + f_a(x^k) \big ] \nonumber \\
&+ 2c\|z^i - x^i\|^2. \label{eq:proof_N_opt1}
\end{align}
where $f(x^k,(z^i,x^{-\{k,i\}})) = \sum_{t \in H} x^{kt} p^t ( \sum_{k \in I, k \neq i} x^{kt} + z^{it} + x^{0t} )$, for all $k \in I$, $k \neq i$ due to \eqref{eq:agent_payoff}, , encoding the fact that the decision vector $z^i$ of agent $i$ appears also in the terms with $k \neq i$. By $x^{-\{k,i\}}$ we mean the elements of $x$ but for the ones corresponding to agents $k$ and $i$.  By $f_0(z^i,x^{-i})$ we imply $f_0(x^1,\ldots,x^{i-1},z^i,x^{i+1},\ldots,x^m) = \sum_{t \in H} x^{0t} p^t (  \sum_{k \in I, k \neq i} x^{kt} +  z^{it} +  x^{0t} )$.
We then have the following result, adapted to the notation of the current paper, due to Corollary 1 of \cite{Deori_etal_2016b}.
\begin{proposition} \label{prop:CorLuca}
(Corollary 1 of \cite{Deori_etal_2016b}) Under Assumption \ref{ass:feas}.a),  the set of minimizers of $\mathcal{P}_a$ coincides with the set of fixed points of the mapping $\widetilde{T}_a$.
\end{proposition}

\begin{proposition} \label{prop:Nash_opt}
Under Assumption \ref{ass:feas}.a), the set of Nash equilibria $N$, and minimizers of $\mathcal{P}_a$ coincide, i.e.,
\begin{align}
N = \arg \min_{\{x^i \in X^i\}_{i \in I}}  f_0(x) + \sum_{i \in I} \big [ f(x^i,x^{-i}) + f_a(x^i) \big ]. \label{eq:prop_Nash_opt}
\end{align}
\end{proposition}
\begin{proof}
By the definition of $\widetilde{T}^i_a(x)$ in \eqref{eq:proof_N_opt1} we have that
\begin{align} \label{eq:proof_N_opt2}
\widetilde{T}^i_a(x) = \arg &\min_{z^i \in X^i} \Big [ \sum_{t \in H} z^{it} p^t \Big ( \sum_{\substack{j \in I\\ j \neq i}} x^{jt} + z^{it} + x^{0t}  \Big ) \nonumber \\
&+ \sum_{t \in H} z^{it} p^t
 \Big ( \sum_{\substack{k \in I\\ k \neq i}} x^{kt} + x^{0t} \Big )  + \sum_{t \in H} p^t (z^{it})^2 \Big ] \nonumber \\
&+ 2c \|z^i - x^i\|^2,
\end{align}
where the first term in the summation corresponds to $f(z^i,x^{-i})$ as defined in \eqref{eq:agent_payoff}, the second term corresponds to $f_0(z^i,x^{-i}) + \sum_{k \in I, k \neq i} f(x^k,(z^i,x^{-\{k,i\}}))$ where all terms that do not depend on the decision vector $z^i$ have been dropped as the leave the minimizer unaffected, and the third term is $f_a(z^i)$ ($f_a(x^k)$ is constant and has been dropped).
Rearranging terms, we obtain
\begin{align} \label{eq:proof_N_opt3}
&\widetilde{T}^i_a(x) = \arg \min_{z^i \in X^i} 2\sum_{t \in H} z^{it} p^t \Big ( \sum_{\substack{j \in I \\ j \neq i}} x^{jt} +z^{it}  + x^{0t} \Big ) \nonumber \\
&~~~~~~~~~~~~~~~~~~~~~~~~~~~~~~~~+ 2c \|z^i - x^i\|^2 \nonumber \\
&= \arg \min_{z^i \in X^i} f(z^i,x^{-i}) + c \|z^i - x^i\|^2 = \widetilde{T}^i(x),
\end{align}
where in the second equality we used \eqref{eq:agent_payoff} and rescaled the objective by a factor of 2, since this does not affect the resulting minimizer. The last equality follows from the definition of $\widetilde{T}^i$ in \eqref{eq:Ttilde_i}.
Equation \eqref{eq:proof_N_opt3} implies that $\widetilde{T}^i_a$ and $\widetilde{T}^i$ are identical. The latter, together with Corollary \ref{cor:N_FTilde} and Proposition \ref{prop:CorLuca}, concludes the proof. \hfill $\Box$
\end{proof}

If we impose also Assumption \ref{ass:feas}.b), the objective function in \eqref{eq:opt_Pa} becomes strictly convex due to the presence of the auxiliary term. Therefore, it admits a unique minimizer and, as a result of Proposition \ref{prop:Nash_opt}, the game of Section \ref{sec:secIIB} admits a unique Nash equilibrium. By Corollary \ref{cor:N_FTilde} this in turn implies that the mapping $\widetilde{T}$ has a unique fixed-point. The uniqueness of the Nash equilibrium is due to \eqref{eq:prop_Nash_opt}, which relies on the particular structure of the objective functions in \eqref{eq:agent_payoff}; for general convex pay-off functions \eqref{eq:prop_Nash_opt}, however, this might not be the case.

The interpretation of \eqref{eq:opt_Pa} is that the auxiliary term acts like a variance penalty in regularization methods (similar to overfitting prevention in regression), promoting least norm solutions, thus implicitly enforcing uniformity in the agents' decisions, and shall not be related to quadratic penalty terms in augmented Lagrangian methods. The relative importance of this term becomes negligible as the number of agents increases.

\subsection{Price of anarchy} \label{sec:secIIIC}
In this subsection we show that as the number of agents increases, the Nash equilibrium of the game in Section \ref{sec:secIIB} achieves the social welfare optimum.

For our analysis we assume that the price coefficients $\{p^t\}_{t \in H}$ are deterministic quantities satisfying Assumption \ref{ass:feas}.b), whereas the consumption level $\gamma^i$, $i \in I$ in \eqref{eq:prob_SO_gamma} and the upper and lower limits in \eqref{eq:prob_SO_limits} are random variables, extracted according to a given probability distribution. We impose the following assumption on the infinite sequence of random vectors $\{\gamma^i,\ \underline{x}^{i},\ \overline{x}^{i}\}_{i\geq 1}$, where $\underline{x}^{i}=[\underline{x}^{i0},\ldots,\underline{x}^{i(h-1)}]$, $\overline{x}^{i}=[\overline{x}^{i0},\ldots,\overline{x}^{i(h-1)}]$.

\begin{assumption} \label{ass:gamma}
Let $\{\gamma^i,\ \underline{x}^{i},\ \overline{x}^{i}\}_{i \geq 1}$ be an infinite sequence of random vectors on a probability space $(\Omega, \mathcal{F}, \mathbb{P})$\footnote{Note that if $\{\gamma^i,\ \underline{x}^{i},\ \overline{x}^{i}\}$, $i\geq 1$, is defined on a given set, by $\mathbb{P}$ we denote the probability measure induced on the infinite cartesian product of these sets. For more details on the mathematical construction of such a measure the reader is referred to \cite{Vidyasagar_1997} (Section 2.4.1, p. 29). }. We assume that
\begin{enumerate}
\item $\{\gamma^i,\ \underline{x}^{i},\ \overline{x}^{i}\}_{i \geq 1}$ are {a sequence of independent and identically distributed (i.i.d.) random vectors.}
\item $\gamma^1$ is a positive random {variable}, while $\underline{x}^{1},\ \overline{x}^{1}$ are non-negative random vectors.
\item $\mathbb{E}[\gamma^1] < \infty$ and $\mathbb{E}[(\gamma^1)^2] < \infty$, where $\mathbb{E}[\cdot]$ denotes the expectation operator associated with the probability measure $\mathbb{P}$.
\end{enumerate}
\end{assumption}
Due to the i.i.d. requirement of Assumption \ref{ass:gamma}.a), the statement of part b) would also hold for all $\gamma^i$, and $\{\underline{x}^{i},\ \overline{x}^{i}\}$, $i\geq1$. By Assumptions \ref{ass:gamma}.a)-b), $\mathbb{E}[\gamma^i] > 0$, for any $i \geq 1$.
We employ the following law of large numbers type of argument, and write that an event holds ($\mathbb{P}$-a.s.) when it holds with probability one with respect to $\mathbb{P}$.

\begin{theorem}[\cite{Shiryaev}, Chapter IV, \S3, Theorem 3] \label{thm:limit}
{Let $\{y^j\}_{j \geq 1}$ be a sequence of i.i.d. random variables such that $\mathbb{E}[|y^1|] < \infty$. For any given index set $J_m$ with cardinality $|J_m| = m$, we then have that
\begin{align}\label{eq:limit}
&\lim_{m \to \infty} \frac{1}{m} \sum_{j \in J_m} y^j = \mathbb{E}[y^1],~~ (\text{$\mathbb{P}$-a.s.})
\end{align}
}
\end{theorem}

Consider any given index set $H$ with $|H| = h$, $h \geq 1$, and let $y^t \in \mathbb{R}$, $y^t \geq 0$, for all $t \in H$. Let also $\bar{y} \in \mathbb{R}$ such that $\sum_{t \in H} y^t = \bar{y}$. Due to norm equivalence we have that $(\| y \|_1/\sqrt{h}) \leq \| y \|_2 \leq \| y \|_1$, where $y = (y^1,\ldots,y^{h})$, i.e.,
\begin{align}
\frac{\bar{y}^2}{h} \leq \sum_{t \in H} (y^t)^2 \leq \bar{y}^2. \label{eq:norms}
\end{align}
which we exploit in the proof of Theorem \ref{thm:limit_Nash_opt}. Denote by
$F^m(x) = f_0(x) + \sum_{i \in I}  f(x^i,x^{-i})$
the objective function of $\mathcal{P}$, and let
$F^m_a(x) = \sum_{i \in I} f_a(x^i)$.
The objective function of $\mathcal{P}_a$ in \eqref{eq:opt_Pa} can be thus written as $F^m(x) + F^m_a(x)$. We introduce the superscript $m$ in our notation to emphasize the fact that the relevant objective functions correspond to a set-up of $m$ agents, since in the sequel we will let $m$ tend to infinity.
Notice that, for any $x \in X$,

\begin{align}
F^m(x) &= \sum_{t \in H} p^t \Big ( \sum_{i \in I} x^{it} + x^{0t} \Big )^2 \nonumber \\ &\geq \underline{p} \sum_{t \in H} \Big ( \sum_{i \in I} x^{it} + x^{0t} \Big )^2 \nonumber \\ &\geq \underline{p} \sum_{t \in H} \Big ( \sum_{i \in I} x^{it} \Big )^2 \geq \underline{p} \frac{\Big ( \sum_{i \in I} \gamma^i \Big )^2}{h} >0, \label{eq:aux_eq}
\end{align}
where the first inequality is obtained by setting $\underline{p} = \min_{t \in H} p^t$, and the second one by ommitting the non-negative term $x^{0t}$. To see the
third inequality notice that $\sum_{t \in H} \left( \sum_{i \in I} x^{it} \right) = \sum_{i \in I} \left( \sum_{t \in H} x^{it} \right) = \sum_{i \in I} \gamma^i$. The desired inequality follows then by the left-hand side of \eqref{eq:norms} with $\sum_{i \in I} x^{it}$, $\sum_{i \in I} \gamma^i$ in place of $y^t$ and $\bar{y}$, respectively. The last inequality is strict, due to the fact that $\underline{p} > 0$ ($H$ is a finite set) as a result of Assumption \ref{ass:feas}.b), and the fact that $\gamma^i > 0$, for all $i \geq 1$, due to Assumption \ref{ass:gamma}.a).

By \cite{Koutsoupias_Papadimitriou_1999}, we have the following definition for the so called \emph{price of anarchy}, which has mainly appeared in the computer science literature, mostly focused on problems with discrete decision variables.
\begin{definition} \label{def:PoA}
For a given $m$, $F^m(x_a^\star)/F^m(x^\star)$ is defined as the price of anarchy for the game in Section \ref{sec:secIIB}.
\end{definition}
Note that according to the discussion below Proposition \ref{prop:Nash_opt} the game under study admits a unique Nash equilibrium. In the opposite case, the numerator of the ratio defined as the price of anarchy shall be replaced by $\max_{x \in N} F^m(x)$, where $N$ is defined as the set of Nash equilibria, to account for the worst-case value achieved by a Nash equilibrium.

\begin{theorem} \label{thm:limit_Nash_opt}
Consider Assumptions \ref{ass:feas} and \ref{ass:gamma}. Let $x^\star \in X$, $x_a^\star \in X$ be any minimizer of $\mathcal{P}$ and $\mathcal{P}_a$, respectively. We then have that
\begin{align}
\lim_{m \to \infty} \frac{F^m(x_a^\star)}{F^m(x^\star)} = 1,~~ (\text{$\mathbb{P}$-a.s.}), \label{eq:limit_Nash_opt}
\end{align}
where $F^m(x^\star) > 0$, i.e., the price of anarchy tends to $1$.
\end{theorem}
\begin{proof}
Let $x, x_a \in X$ be feasible solutions, possibly different, of $\mathcal{P}$ and $\mathcal{P}_a$, respectively.
By the definition of $F^m$, $F^m_a$, and since $F^m(x)>0$ for any $x \in X$, we have that
\begin{align}
\frac{F^m_a(x_a)}{F^m(x)} = \frac{\sum_{t \in H} p^t \sum_{i \in I} (x^{it}_a)^2}{\sum_{t \in H} p^t \left( \sum_{i \in I} x^{it} + x^{0t} \right)^2}. \label{eq:proof_thm1}
\end{align}
Let $\overline{p}=\max_{t \in H} p^t$ and $\underline{p} = \min_{t \in H} p^t > 0$, where the inequality is strict due to Assumption \ref{ass:feas}.b). We have that
\begin{align}
\frac{F^m_a(x_a)}{F^m(x)} \leq \frac{ \overline{p} \sum_{t \in H} \sum_{i \in I} (x^{it}_a)^2}{ \underline{p} \sum_{t \in H} \left( \sum_{i \in I} x^{it} + x^{0t} \right)^2}. \label{eq:proof_thm2}
\end{align}

Since $x^{it}_a$ is feasible for $\mathcal{P}_a$, we have that $\sum_{t \in H} x^{it}_a = \gamma^i$, for all $i \in I$. By the right-hand side of \eqref{eq:norms} with $x^{it}$, $\gamma^i$ in place of $y^t$ and $\bar{y}$, respectively, we obtain that
\begin{align}
\sum_{t \in H} (x^{it}_a)^2 \leq (\gamma^i)^2, \text{ for all } i \in I. \label{eq:proof_thm3}
\end{align}
By the derivation of \eqref{eq:aux_eq}, we obtain that
\begin{align}
\sum_{t \in H} \left( \sum_{i \in I} x^{it} + x^{0t} \right)^2 \geq \frac{\Big ( \sum_{i \in I} \gamma^i \Big )^2}{h}. \label{eq:proof_thm4}
\end{align}
Employing \eqref{eq:proof_thm3}, \eqref{eq:proof_thm4}, and by exchanging the summation order in the numerator of \eqref{eq:proof_thm2}, we have that
\begin{align}
\frac{F^m_a(x_a)}{F^m(x)} \leq \frac{ \overline{p} h \sum_{i \in I} (\gamma^i)^2}{\underline{p} \Big ( \sum_{i \in I} \gamma^i \Big )^2} = \frac{ \overline{p} h \frac{\sum_{i \in I} (\gamma^i)^2}{m}}{ \underline{p} m \Big (  \frac{\sum_{i \in I} \gamma^i}{m} \Big )^2}. \label{eq:proof_thm5}
\end{align}

Applying Theorem \ref{thm:limit} twice, once with $\gamma^i$ and once with $(\gamma^i)^2$ in place of $y^i$, we have that $\mathbb{P}$-a.s.
\begin{align}
\lim_{m\rightarrow\infty} \tfrac{\sum_{i \in I} \gamma^i}{m}
&= \mathbb{E}[\gamma^1] \nonumber \\
\lim_{m\rightarrow\infty} \tfrac{\sum_{i \in I} (\gamma^i)^2}{m}
&= \mathbb{E}[(\gamma^1)^2]  \nonumber
\end{align}
However, since $\mathbb{E}[\gamma^1] > 0$ and $\mathbb{E}[(\gamma^1)^2]/\big ( \mathbb{E}[\gamma^1] \big )^2 < \infty$ due to Assumption \ref{ass:gamma}.c),
\begin{align}
\lim_{m \to \infty} \frac{ \overline{p} h \frac{\sum_{i \in I} (\gamma^i)^2}{m}}{ \underline{p} m \Big (  \frac{\sum_{i \in I} \gamma^i}{m} \Big )^2} = 0.~~ (\text{$\mathbb{P}$-a.s.}) \label{eq:proof_thm7}
\end{align}
Therefore, since \eqref{eq:proof_thm5} holds for any $\{\gamma^i\}_{i \in I}$, we have that
\begin{align}
\lim_{m\rightarrow\infty} \frac{F^m_a(x_a)}{F^m(x)} = 0,~~ (\text{$\mathbb{P}$-a.s.}) \label{eq:proof_thm8}
\end{align}

Let now $x^\star, x^\star_a \in X$ denote an optimal solution of $\mathcal{P}$ and $\mathcal{P}_a$, respectively. By optimality of $x^\star_a$ we thus have that
\begin{align}
F^m(x^\star_a) + F^m_a(x^\star_a) \leq F^m(x^\star) + F^m_a(x^\star). \label{eq:proof_thm9}
\end{align}
Rearranging the terms in \eqref{eq:proof_thm9}, and since $F^m(x^\star) > 0$ (see discussion above Theorem \ref{thm:limit_Nash_opt}), we obtain
\begin{align}
\frac{F^m(x^\star_a) - F^m(x^\star)}{F^m(x^\star)} &\leq \frac{F^m_a(x^\star) - F^m_a(x^\star_a)}{F^m(x^\star)} \leq \frac{F^m_a(x^\star)}{F^m(x^\star)}, \label{eq:proof_thm10}
\end{align}
where the last inequality is due to the fact that $F^m_a(x^\star_a) \geq 0$. Since \eqref{eq:proof_thm8} holds for any $x, x_a \in X$, it will also hold for $x = x_a = x^\star$. Therefore, \eqref{eq:proof_thm8} and \eqref{eq:proof_thm10} lead to
\begin{align}
\lim_{m\rightarrow\infty} \frac{F^m(x^\star_a) - F^m(x^\star)}{F^m(x^\star)} = 0,~~ (\text{$\mathbb{P}$-a.s.}) \label{eq:proof_thm11}
\end{align}
which in turn implies \eqref{eq:limit_Nash_opt}, thus concluding the proof. \hfill $\Box$
\end{proof}

Informally speaking, the price of anarchy quantifies the gap between the social optimum and the value of the non-cooperative game; Theorem \ref{thm:limit_Nash_opt} implies that this gap tends to zero as the number of agents increases.
\begin{remark}
In Theorem \ref{thm:limit_Nash_opt} we used the fact that the parameters that give rise to a heterogeneous vehicle population are random and satisfy Assumption \ref{ass:gamma}. This offers a more flexible framework to model agents' heterogeneity, e.g., encoding prior information on their distribution,
and is in line with the mean-field game theoretic approach adopted in \cite{Caines_etal_2007} for unconstrained, continuous time quadratic games. However, if instead of Assumption \ref{ass:gamma} we assume that for all $i \in I$, $\gamma^i \in [\underline{\gamma},\overline{\gamma}]$ for given deterministic quantities $\underline{\gamma},\overline{\gamma} \in \mathbb{R}$ (similarly for $\underline{x}^i,\overline{x}^i$) with $\underline{\gamma} > 0$, the result of Theorem \ref{thm:limit_Nash_opt} remains valid not probabilistically, but for all $\gamma^i \in [\underline{\gamma},\overline{\gamma}]$. In particular, the proof remains unchanged but for the following modifications: The inequalities in \eqref{eq:proof_thm5} shall be replaced by
\begin{align}
\frac{F^m_a(x_a)}{F^m(x)} \leq \frac{ \overline{p} h \sum_{i \in I} (\gamma^i)^2}{\underline{p} \Big ( \sum_{i \in I} \gamma^i \Big )^2} \leq \frac{\overline{p} h m \overline{\gamma}^2 }{\underline{p} m^2 \underline{\gamma}^2} \leq \frac{\overline{p} h  \overline{\gamma}^2 }{\underline{p} m \underline{\gamma}^2} , \label{eq:proof_thm12}
\end{align}
where the numerator of the second inequality follows from $\sum_{i \in I} (\gamma^i)^2 \leq m \overline{\gamma}^2$ and the denominator from $( \sum_{i \in I} \gamma^i )^2 \geq m^2 \underline{\gamma}^2 > 0$. Equation \eqref{eq:proof_thm12} leads to $\lim_{m\rightarrow\infty} (F^m_a(x_a)/F^m(x)) = 0$ and from \eqref{eq:proof_thm9} the proof of Theorem \ref{thm:limit_Nash_opt} remains unchanged, with the relevant statements holding robustly for all $\gamma^i \in [\underline{\gamma},\overline{\gamma}]$, $i \in I$, instead of $\mathbb{P}$-a.s.
\end{remark}

Note that the aggregate quantity $\frac{1}{m} \sum_{i \in I} x^{it}$ exhibits the same behaviour with the corresponding objective functions of $\mathcal{P}$ and $\mathcal{P}_a$ in Theorem \ref{thm:limit_Nash_opt}, since under Assumption \ref{ass:gamma}.b) the latter are strictly convex with respect to the agents aggregate.

To illustrate the result of Theorem \ref{thm:limit_Nash_opt}, we performed a numerical investigation parametric with respect to the number of agents $m$. We considered a time horizon $h=12$, and price coefficients $(p^0,\ldots,p^{h-1}) = (0.1, 1, 1.9, 2.8, 3.7, 4.6, 5.5, 6.4, 7.3, 8.2, 9.1, 10)$. For simplicity we assumed that the probability mass is concentrated to the lower and upper limits $\underline{x}^{it} = 0$ and $\overline{x}^{it} = 1$ for all $i \in I$, $t \in H$ (assuming normalized charging rates) that are effectively being treated as deterministic, whereas the charging levels $\gamma^i$, $i \in I$, were extracted in an i.i.d. fashion from a uniform distribution with support $[0, 12]$. We consider a zero non-PEV demand, i.e., $\hat{x}^{0t} = 0$ for all $t \in H$ (see Section \ref{sec:secIIA} for a definition of $\hat{x}^{0t}$). For each $m$, we performed 100 multi-extractions of $\{\gamma^i\}_{i \ge 1}$, and calculated the average of the ratio $(F^m(x^\star_a) - F^m(x^\star))/F^m(x^\star)$. As shown in Figure \ref{fig:error_limit}, and following \eqref{eq:limit_Nash_opt}, this ratio tends to zero as the number of agents increases for every set of heterogeneity parameters, but not necessarily in a monotone way.
Note that $x^\star, x^\star_a$ depend on the extracted $\{\gamma^i\}_{i \ge1}$; however, we suppress this dependence in the notation for simplicity.

Figure \ref{fig:valleyFill} investigates the case of a non-zero normalized non-PEV demand, i.e., $x^{0t}/m = \hat{x}^{0t} \neq 0$ (green), and considers
the normalized total consumption profile $(1/m)(\sum_{i\in I}x^{it} +x^{0t}) = (1/m) \sum_{i\in I}x^{it} + \hat{x}^{0t}$ obtained by solving problem $\mathcal{P}$ (blue) and problem $\mathcal{P}_a$ (red). Here we solved those problems by means of the iterative algorithm proposed in \cite{Deori_etal_2016b}, but other decentralized algorithms could be employed, e.g., \cite{Gan_etal_2013}. Both solutions have the so called valley filling property, i.e., the PEV consumption tends to compensate for the over night drop in the non-PEV consumption. By comparison of the figure panels, as $m$ increases the consumption corresponding to the Nash equilibrium tends to the social optimum, as expected by the discussion below Remark 1.

\begin{figure}[t!]
\centering
\includegraphics[scale=0.51]{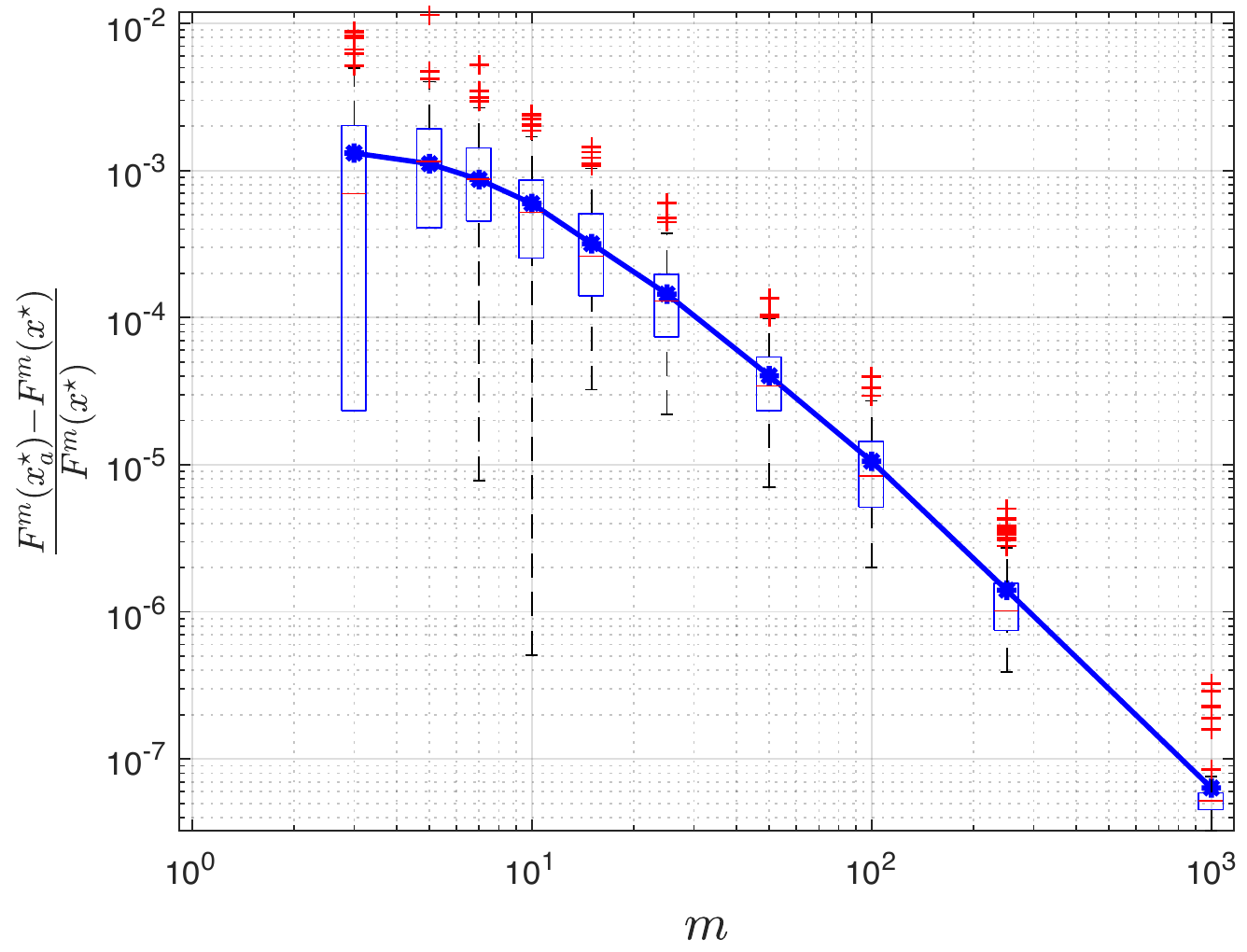}
\caption{Relative error $(F^m(x^\star_a) - F^m(x^\star))/F^m(x^\star)$; ``Blue stars'' correspond to the average value across 100 multi-extractions of $\{\gamma^i\}_{i \ge1}$ from a uniform distribution and $\underline{x}^{it} = 0$, $\overline{x}^{it} = 1$ for all $i \in I$, $t \in H$, while for each $m$ boxplots show the distribution of the relative error for the different parameter extractions. On each box, the ``red line'' indicates the median, and the bottom and top edges of the box indicate the 25th and 75th percentiles, respectively. The whiskers extend to the most extreme data points not considered outliers, and the outliers are plotted individually using the ``+'' symbol.}
\label{fig:error_limit}
\end{figure}

\begin{figure}[t!]
\centering
\subfloat[$m=5$]{\includegraphics[scale=0.54]{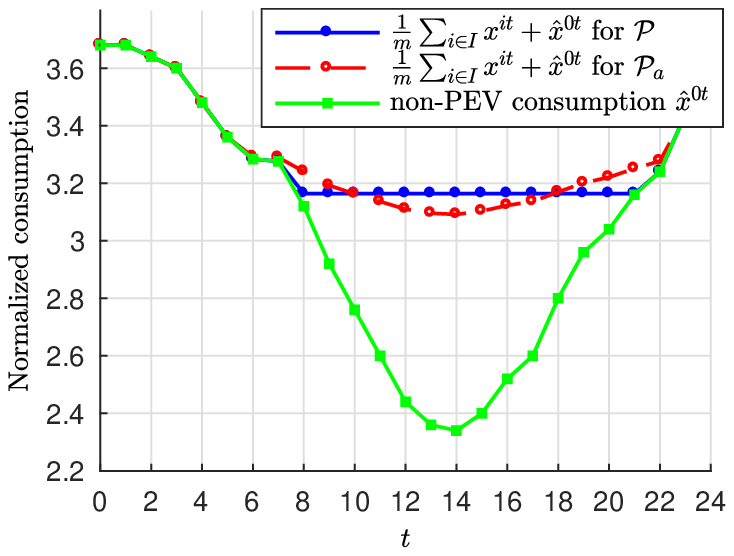}}
\subfloat[$m=100$]{\includegraphics[scale=0.54]{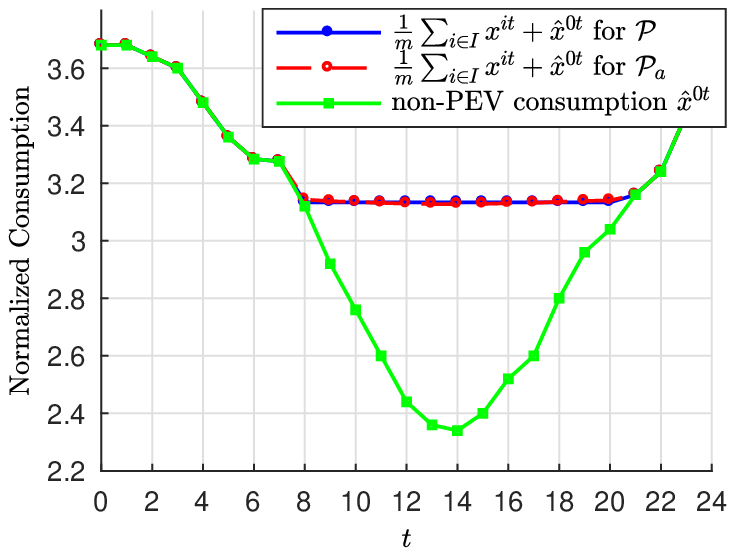}}
\caption{Normalized total consumption profile $(1/m) \sum_{i\in I}x^{it} +\hat{x}^{0t}$ obtained by solving $\mathcal{P}$ (blue) and $\mathcal{P}_a$ (red); Normalized non-PEV consumption $\hat{x}^{0t}$ shown in green. As $m$ increases these profiles tend to coincide.}
\label{fig:valleyFill}
\end{figure}

\section{Effect of heterogeneity} \label{sec:secIV}
Define the random vectors $\{\xi^i\}_{i \geq 1}=\{\gamma^i,\ \underline{x}^{i},\ \overline{x}^{i}\}_{i \geq 1}$.
For the results of this section we assume that $\{\xi^i\}_{i \geq 1}$ are extracted from a discrete probability distribution. \begin{assumption} \label{ass:gamma_discrete}
Let $\{\xi^i\}_{i \geq 1}=\{\gamma^i,\ \underline{x}^{i},\ \overline{x}^{i}\}_{i \geq 1}$ be an infinite sequence of positive, i.i.d. random variables on a \emph{discrete} probability space $(\Omega, \mathcal{F}, \mathbb{P})$. We assume that $\mathbb{P}$ is supported on $n_\xi$ masses located at $\bar{\xi}^\ell=[\gamma^\ell,\ {\underline{x}^\ell},\ {\overline{x}^\ell}]$, $\ell \in L$, where
$L = \{1, \ldots, n_\xi\}$, i.e., $\sum_{\ell \in L} \mathbb{P} \{ \xi = \bar{\xi}^\ell\} = 1$, for any $\xi \in \Omega$.
\end{assumption}

\subsection{Abstraction in homogeneous groups} \label{sec:secIVA}
In this subsection we focus on a finite number of agents and show that, either when solving $\mathcal{P}$ or $\mathcal{P}_a$, the decision vectors corresponding to agents that form a homogeneous group are identical, i.e., identical vehicles have the same charging profile. This naturally provides a way to abstract the overall problem, involving a possibly high number of agents and hence decision vectors, to a problem of smaller size where we only have one decision vector per group of homogeneous agents.

For any $m \geq 1$, for all $i \in I$, denote by $\sum_{i \in I} \mathds{1}_{\{\xi^i = \bar{\xi}^\ell\}}$ the number of agents that form a homogeneous group with parameter $\bar{\xi}^\ell$, where $\mathds{1}_{\{\xi^i = \bar{\xi}^\ell\}}$ is an indicator function that is 1 if $\xi^i = \bar{\xi}^\ell$ and 0 otherwise. For all $\ell \in L$, denote by $I^\ell = \{i \in I:~ \xi^i = \bar{\xi}^\ell \}$ the set of indices corresponding to agents belonging to the same homogeneous group. Note that for the single agent case (i.e., $m=1$) one of the sets $I^\ell$, $\ell \in L$, is singleton and all the others are empty. This implies that there is only one term in the square in $\bar{F}^m$ below.

Let $\bar{x}^\ell = [ \bar{x}^{\ell0}, \ldots, \bar{x}^{\ell(h-1)}]^\top \in \mathbb{R}^{|H|}$, $\ell \in L$, $\bar{x} = (\bar{x}^1,\ldots,$ $ \bar{x}^{n_\xi})$, $\bar{X} = X^1 \times \ldots \times \bar{X}^{n_\xi}$,
and consider the following variant of $\mathcal{P}$, where we only consider one decision vector per group of homogeneous agents.
\begin{align}
\bar{\mathcal{P}}:~ \min_{\bar{x} \in \bar{X}} \bar{F}^m(\bar{x}), \label{eq:opt_Pbar}
\end{align}
where
$\bar{F}^m(\bar{x}) = \sum_{t \in H} p^t ( \sum_{\ell \in L} \sum_{i \in I} \mathds{1}_{\{\xi^i = \bar{\xi}^\ell\}} \bar{x}^{\ell t} + x^{0t} )^2$, 
and for all $\ell \in L$,
\begin{align}
\bar{X}^\ell = \big \{ \bar{x}^\ell \in \mathbb{R}^{|H|}:~ & \sum_{t \in H} \bar{x}^{\ell t} = {\gamma}^\ell \text{ and } \nonumber \\
&\bar{x}^{\ell t} \in [\underline{x}^\ell, \overline{x}^\ell], \text{ for all } t \in H \big \}. \label{eq:agent_con_bar}
\end{align}
Let also $\bar{\mathcal{P}}_a$ denote the variant of $\mathcal{P}$, defined similarly to $\bar{\mathcal{P}}$ with the difference that its objective function is the sum of the objective function in \eqref{eq:opt_Pbar} and the term $\sum_{t \in H} p^t ( \bar x^{\ell t} )^2$.

\begin{proposition} \label{prop:group}
Consider Assumptions \ref{ass:feas}.a) and \ref{ass:gamma_discrete}. Let $\bar{x}^\star \in \bar{X}$, $\bar{x}^\star_a \in \bar{X}$ be any minimizer of $\bar{\mathcal{P}}$ and $\bar{\mathcal{P}}_a$, respectively. For all $\ell \in L$, let
\begin{align}
x^{i,\star} &= \bar{x}^{\ell,\star}, \text{ for all } i \in I^\ell, \label{eq:minP} \\
x^{i,\star}_a &= \bar{x}^{\ell,\star}_a, \text{ for all } i \in I^\ell, \label{eq:minPa}
\end{align}
Vectors $x^\star = (x^{1,\star}, \ldots, x^{m,\star})$ and $x^\star_a = (x^{1,\star}_a, \ldots,$ $x^{m,\star}_a)$ are minimizers of $\mathcal{P}$ and $\mathcal{P}_a$, respectively.
\end{proposition}
\begin{proof}
For all $\ell \in L$, for all $i \in I^\ell$, $X^i = \bar{X}^\ell$. Therefore, since $\bar{x}^{\star}$ is optimal for $\bar{\mathcal{P}}$, it will be also feasible, i.e., $\bar{x}^{\ell,\star} \in \bar{X}^\ell$, for all $\ell \in L$. The last statement, together with \eqref{eq:minP}, leads to $x^{i,\star} \in X^i$, for all $i \in I$, which in turn implies that $x^\star$ is a feasible solution for $\mathcal{P}$. Via an analogous argument it can be shown that $x^\star_a$ is a feasible solution for $\mathcal{P}_a$.

By the definition of $\bar{F}^m$ we have that
\begin{align}
\bar{F}^m(\bar{x}^\star)
&= \sum_{t \in H} p^t \Big ( \sum_{\ell \in L} \sum_{i \in I} \mathds{1}_{\{\xi^i = \bar{\xi}^\ell\}} \bar{x}^{\ell t,\star} + x^{0t}\Big )^2 \nonumber \\
&= \sum_{t \in H} p^t \Big ( \sum_{\ell \in L} \sum_{i \in I^\ell} x^{it,\star} + x^{0t}\Big )^2 \nonumber \\
&= \sum_{t \in H} p^t \Big ( \sum_{i \in I} x^{it,\star} +x^{0t}\Big )^2 = F^m(x^\star), \label{eq:cost_equal}
\end{align}
where the third equality is due to
\eqref{eq:minP}, and the last one is due to \eqref{eq:agent_payoff}.
Let $x = (x^1,\ldots,x^m) \in X$ be an arbitrary feasible solution of $\mathcal{P}$, i.e., $x^i \in X^i$ for all $i \in I$, and consider
$\bar{x}^\ell = (1/n^\ell) \sum_{i \in I^\ell} x^i$, for all $\ell \in L$. For $\ell \in L$, since $\bar{x}^\ell$ is a convex combination of $\{x^i \in X^i \}_{i \in I^\ell}$, $X^i = \bar{X}^\ell$ for all $i \in I^\ell$ and $\bar{X}^\ell$ is convex, $\bar{x}^\ell \in \bar{X}^\ell$. Hence,
\begin{align}
\!\bar{F}^m(\bar{x}^\star) \!&\leq\! \bar{F}^m(\bar{x}) = \sum_{t \in H} p^t \Big ( \sum_{\ell \in L} \sum_{i \in I} \!\!\!\mathds{1}_{\{\xi^i = \bar{\xi}^\ell\}} \bar{x}^{\ell t}\! +\! x^{0t}\Big )^2 \nonumber \\
&\leq \sum_{t \in H} p^t \Big ( \sum_{\ell \in L} \sum_{i \in I^\ell} x^{it} + x^{0t}\Big )^2 \nonumber \\
&= \sum_{t \in H} p^t \Big ( \sum_{i \in I} x^{it} + x^{0t} \Big )^2 = F^m(x), \label{eq:proof_group1}
\end{align}
where the first inequality is due to optimality of $\bar{x}^\star$ for $\bar{\mathcal{P}}$, whereas the second one is due to
convexity of $\bar{F}^m$ and the fact that it is quadratic with respect to $\bar{x}$.
By \eqref{eq:cost_equal} and \eqref{eq:proof_group1}, we have that $F^m(x^\star) \leq F^m(x)$. Since $x \in X$ was arbitrary, $x^\star$ is optimal for $\mathcal{P}$.
To show that $x^\star_a$ is optimal for $\mathcal{P}_a$ we follow the same derivation with \eqref{eq:cost_equal} and \eqref{eq:proof_group1}, appending to $\bar{F}^m$ the term $\sum_{\ell \in L} \sum_{t \in H} p^t ( x^{\ell t} )^2$. \hfill $\Box$
\end{proof}

Proposition \ref{prop:group} implies that it suffices to solve $\bar{\mathcal{P}}$ (similarly for $\bar{\mathcal{P}}_a$), which involves fewer decision variables compared to $\mathcal{P}$, and then construct a minimizer of $\mathcal{P}$ by means of the assignment in \eqref{eq:minP}. Note that \eqref{eq:minP} and \eqref{eq:minPa} enforce the same decision vector to all members of a homogeneous group. It should be noted that the result of Proposition \ref{prop:group} is intuitive; as an effect of the price being agent independent, all agents in a homogeneous group solve exactly the same optimisation problem, thus resulting to the same Nash equilibrium charging strategy.

\subsection{Asymptotic effect of heterogeneity} \label{sec:secIVB}
Theorem \ref{thm:limit_Nash_opt} shows that the ratio between the optimal values of $\mathcal{P}$ and $\mathcal{P}_a$ tends to one as $m$ tends to infinity, for almost any $\{\xi^i\}_{i \geq 1}$, however, their individual values may change for different values of $\{\xi^i\}_{i \geq 1}$. For the case of a discrete probability distribution, we show in the following theorem that this is not the case and, as the number of agents tends to infinity, the optimal value of $\mathcal{P}$ (and hence the one of the associated game) tends to a deterministic quantity, i.e., variability averages out as the number of agents increases. For that particular subclass of problems and distributions, this result provides support to hypothesis $H_3'$ in \cite{Caines_etal_2007}.

\begin{theorem} \label{thm:heter}
Consider Assumptions \ref{ass:feas}.a) and \ref{ass:gamma_discrete}. For any $m \geq 1$, let $x^\star$, $\bar{x}^\star$ be any minimizer of $\mathcal{P}$ and $\bar{\mathcal{P}}$, respectively. We then have that
\begin{align}
\lim_{m \to \infty} &\frac{F^m(x^\star)}{m^2} \nonumber \\ &= \sum_{t \in H} p^t \Big ( \sum_{\ell \in L} \mathbb{P}\{\xi = \bar{\xi}^\ell\} \bar{x}^{\ell t, \star} + \hat{x}^{0t} \Big )^2,~~ (\text{$\mathbb{P}$-a.s.}) \label{eq:heter}
\end{align}
\end{theorem}
\begin{proof}
For all $\ell \in L$, by Theorem \ref{thm:limit} with $\mathds{1}_{\{\xi^i = \bar{\xi}^\ell\}}$ in place of $y^i$, and since $\mathbb{E}[\mathds{1}_{\{\xi = \bar{\xi}^\ell\}}] = \mathbb{P}\{\xi = \bar{\xi}^\ell\}$, for all $\xi \in \Omega$,
\begin{align}\label{eq:proof_heter1}
\lim_{m \to \infty} \frac{1}{m} \sum_{i \in I} \mathds{1}_{\{\xi^i = \bar{\xi}^\ell\}} = \mathbb{P}\{\xi = \bar{\xi}^\ell\},~~ (\text{$\mathbb{P}$-a.s.})
\end{align}
By \eqref{eq:cost_equal} we have that $F^m(x^\star) = \bar{F}^m(\bar{x}^\star)$, while by the definition of $\bar{F}^m$ we obtain that
\begin{align}
\frac{F^m(x^\star)}{m^2} = \sum_{t \in H} p^t \Big ( \sum_{\ell \in L} \sum_{i \in I}\!\! \frac{1}{m} \mathds{1}_{\{\xi^i = \bar{\xi}^\ell\}} \bar{x}^{\ell t, \star} +\frac{x^{0t}}{m}\Big )^2\!, \label{eq:proof_heter2}
\end{align}
Since \eqref{eq:proof_heter2} holds for any $\{\xi^i\}_{i \in I}$, for any $m \geq 1$, \eqref{eq:proof_heter1}, \eqref{eq:proof_heter2}, and the fact that $\lim_{m \to \infty} x^{0t}/m = \hat{x}^{0t}$ (see Section \ref{sec:secIIA}), lead to \eqref{eq:heter}, and hence conclude the proof. \hfill $\Box$
\end{proof}

By Theorem \ref{thm:limit_Nash_opt} a similar statement holds for the optimal value of $\mathcal{P}_a$, as this tends to the one of $\mathcal{P}$ as the number of agents increases.
The implication of Theorem \ref{thm:heter} is illustrated in Figure \ref{fig:heter}.
We consider the same set-up with that of Figure \ref{fig:error_limit}, where $\hat{x}^{0t} = 0$ for all $t \in H$, with the difference that the charging levels $\{\gamma^i\}_{i \ge 1}$, $i \in I$, were extracted in an i.i.d. fashion from a discrete uniform distribution in $[0, 12]$, with masses centered uniformly in this interval with spacing $0.01$.
For different values of $m$, we provide the empirical probability distribution of $F^m(x^\star)/m^2$, where $x^\star$ is calculated by solving $\mathcal{P}$. As $m$ increases, the empirical distribution becomes concentrated at a single value of $F^m(x^\star)/m^2$, in agreement with Theorem \ref{thm:heter}.
\begin{figure}[t!]
\centering
\includegraphics[trim=1.9cm 6cm 2cm 7cm, clip=true,scale=0.42]{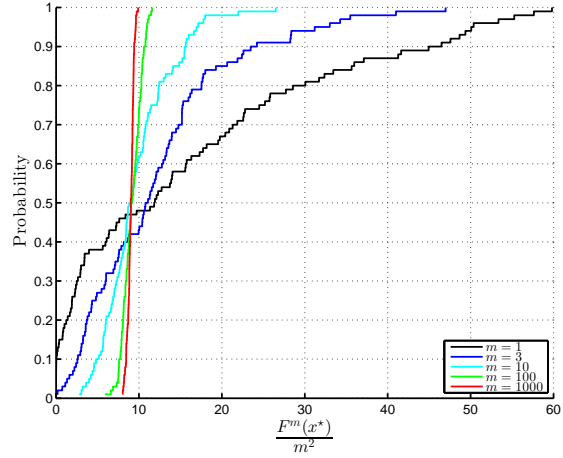}
\caption{Empirical distribution of $F^m(x^\star)/m^2$, constructed by calculating the optimal solution $x^\star$ of $\mathcal{P}$ for 100 multi-extractions of $\{\gamma^i\}_{i \ge1}$ from a discrete uniform distribution and $\underline x^{it}=0$ and $\overline x^{it}=1$ for all $i\in I$, $t\in H$. As $m$ increases the distribution gets concentrated around the quantity in \eqref{eq:heter}.}
\label{fig:heter}
\end{figure}

\section{Concluding remarks} \label{sec:secVI}
We quantified the price of anarchy for a class of PEV charging control games, showing that the limiting case of infinite agent populations the Nash equilibrium achieves the same value with the social welfare optimum for almost any choice of the random heterogeneity parameter. Moreover, in the case where the agents' heterogeneity parameters follow a discrete probability distribution, we provided a systematic way to abstract agents in homogeneous groups and showed that heterogeneity averages out as the number of agents tends to infinity.

Several iterative algorithms for decentralized computation of Nash equilibria could be employed, e.g.,  \cite{Gan_etal_2013,Paccagnan_etal_2016,Deori_etal_2016b}; in \cite{Deori_arxiv} a detailed analysis using the regularized Jacobi algorithm of \cite{Deori_etal_2016b} is provided.
Current work concentrates on relaxing the requirement for an affine price function to allow for a more general class of games like in \cite{Gan_etal_2013}, and on incorporating distribution network models and intertemporal charging costs in our formulation \cite{Caramanis_2012, Caramanis_IEEE}. Moreover, we aim at investigating the effect of heterogeneity in the case where the underlying probability distribution is continuous, while the result of Theorem \ref{thm:heter} could be exploited from a system aggregator's point of view to steer the aggregate value of large fleets of vehicles to a given deterministic quantity.

\section{Acknowledgements}
We would like to thank the anonymous reviewers, and in particular one of them for suggesting Remark 1.
\bibliographystyle{plain}
\bibliography{ref_decentralized_EV}

\vfill

\end{document}